\documentclass[12pt]{amsart}
\usepackage{amsmath,amsthm,amsfonts,amssymb,times,latexsym,mathabx,url}

\newtheorem{theorem}{Theorem}[section]

\newtheorem{lem}[theorem]{Lemma}

\newtheorem*{rem}{Remark}
\numberwithin{equation}{section}

\newcommand{\e}{\varepsilon}
\renewcommand{\o}{\omega}

\renewcommand{\leq}{\leqslant}
\renewcommand{\geq}{\geqslant}
\renewcommand{\d}{\delta}

\makeatletter
\renewcommand{\pmod}[1]{\allowbreak\mkern7mu({\operator@font mod}\,\,#1)}
\makeatother

\renewcommand{\o}{\omega}

\renewcommand{\leq}{\leqslant}
\renewcommand{\geq}{\geqslant}

\newcommand{\M}{{\mathcal{M}}}
\newcommand{\I}{{\mathbb I}}
\renewcommand{\d}{{\delta}}

\renewcommand{\leq}{\leqslant}
\renewcommand{\geq}{\geqslant}
\renewcommand{\o}{\omega}

\numberwithin{equation}{section}

\begin{document} 

\title{Karatsuba's divisor problem and related questions}
\author{Mikhail R. Gabdullin, Vitalii V. Iudelevich, Sergei V. Konyagin}
\date{}
\address{Steklov Mathematical Institute,
Gubkina str., 8, Moscow, Russia, 119991}
\email{gabdullin.mikhail@yandex.ru}
\address{Moscow State University, Leninskie Gory str., 1, Moscow, Russia, 119991}
\email{vitaliiyudelevich@mail.ru}
\address{Steklov Mathematical Institute,
	Gubkina str., 8, Moscow, Russia, 119991}
\email{konyagin@mi-ras.ru}

\begin{abstract}
We prove that
$$
\sum_{p \leq x} \frac{1}{\tau(p-1)} \asymp \frac{x}{(\log x)^{3/2}} \quad \mbox{\text{and}} \quad \sum_{n \leq x} \frac{1}{\tau(n^2+1)} \asymp \frac{x}{(\log x)^{1/2}},
$$ 
where $\tau(n)=\sum_{d|n}1$ is the number of divisors of $n$, and the summation in the first sum is over primes.
\end{abstract}

\date{\today}
\maketitle

\section{Introduction}

In 2004, A.~A. Karatsuba in his seminar ``Analytic number theory and applications'' posed the following problem: find the asymptotics for the sum
$$
\Phi (x) = \sum_{p \leq x} \frac{1}{\tau(p-1)}
$$
as $x\to\infty$, where $\tau(n) = \sum_{d|n} 1$ is the divisor function and the summation is over primes not exceeding $x$. This is a natural ``hybrid'' of the following two classical problems from analytic number theory.

The first one (so-called Titchmarsh divisor problem) is asking for the asymptotic behaviour of the sum 
$$
D(x) = \sum_{p\leq x}\tau(p-1).
$$
It is well-known (see \cite{Titchmarsh}, \cite{Bombieri}, \cite{Linnik}) that
$$
D(x) \sim \frac{\zeta(2)\zeta(3)}{\zeta(6)}\,x, \quad\quad x\to\infty,
$$
where $\zeta(s)$ denotes the Riemann zeta-function. The other problem is to find the asymptotics for the sum 
$$
T(x) = \sum_{n\leq x}\frac{1}{\tau(n)}.
$$
This problem was solved by S.~Ramanujan \cite{Ram1916}, who showed that
\begin{equation}\label{1.1}
T(x) = c_0\frac{x}{\sqrt{\log x}}\left( 1+O\left(\frac{1}{\sqrt{\log x}} \right)\right) ,
\end{equation}
where
$$
c_0 = \frac{1}{\sqrt{\pi}} \prod_{p} \sqrt{p^2 - p} \log \frac{p}{p-1} = 0.5486\ldots
$$

The sum $\Phi(x)$ has been studied before. In the recent work \cite{Iud2022} it was shown that
$$
\Phi(x)\leq 4K\dfrac{x}{(\log x)^{3/2}}+O\left( \frac{x\log\log x}{(\log x)^{5/2}}\right),$$
where
$$
K=\frac{1}{\sqrt{\pi}}\prod_{p}\sqrt{\dfrac{p}{p-1}}\left(p\log\dfrac{p}{p-1} -\dfrac{1}{p-1}\right) = 0.2532\ldots 
$$
Note that the bound $\Phi(x)\ll \frac{x}{(\log x)^{3/2}}$ follows from Corollary 1.2 of \cite{Pol2020}, and can also be derived by arguing similarly to the proof of upper bound in Theorem \ref{th1.2} below. We make the conjecture that
$$
\Phi(x) \sim K\frac{x}{(\log x)^{3/2}}, \quad\quad x\to\infty,
$$
but this is probably hard to show.

In this paper, we prove that the aforementioned upper bound for $\Phi(x)$ is sharp up to a constant.  

\begin{theorem}\label{th1.1}
We have
$$
\Phi (x) \gg \frac{x}{(\log x)^{3/2}}.
$$
\end{theorem}

Thus, we find the correct order of magnitude of the sum $\Phi(x)$: 
$$
\Phi(x)\asymp \frac{x}{(\log x)^{3/2}},
$$ 
which answers the question of Karatsuba in the first approximation. 
 
 \smallskip
 
In addition to $\Phi(x)$, we consider the sum 
$$
F(x) = \sum_{n \leq x} \frac{1}{\tau(n^2+1)},
$$
and establish the following.
 
\begin{theorem}\label{th1.2}
We have
$$ 
F(x)\asymp \frac{x}{(\log x)^{1/2}}.
$$
\end{theorem}

\medskip 
 
Now we discuss the main ideas of the proofs. In the sum $\Phi(x)$, for each prime $p\leq x$, we write $p-1$ in the form $ab$, where $a$ consists of prime factors not exceeding $z$, and $b$ has only prime factors greater than $z$; here $z=x^\e$ for some small fixed $\e>0$. Now $\Phi(x)$ can be rewritten as
\begin{equation}\label{0.2}
\Phi(x) = \sum_{\substack{a \leq x \\ p|a \Rightarrow p \leq z}} \frac{1}{\tau(a)} \sum_{\substack{b \leq \frac{x-1}{a} \\ p|b \Rightarrow p>z \\ ab +1 \text{ prime}}}  \frac{1}{\tau(b)},
\end{equation}
and, since $\tau(b)=O_\e(1)$, we have
\begin{equation*}\label{lower}
\Phi(x)\gg_\e \sum_{a\leq x^\e} \frac{1}{\tau(a)} \sum_{\substack{b \leq \frac{x-1}{a} \\ p|b \Rightarrow p>x^\e \\ ab +1 \text{ prime}}}  1.
\end{equation*}

Using the Brun-Hooley sieve, one can estimate the inner sum from below by a quantity of order
$$ 
\frac{x}{a(\log x)^2}-R(x;a),
$$ 
where the contribution of $R(x;a)$ to the outer sum will be negligible. Combining these estimates, we see that
$$
\Phi(x) \gg_\e \frac{x}{(\log x)^2}\sum_{a\leq x^\e} \frac{1}{a\tau(a)}\gg_\e \frac{x}{(\log x)^{3/2}},
$$
which gives the desired lower bound. Let us note that this argument does not yield any upper bound for $\Phi(x)$, since estimation of the contribution of numbers $a>x^{1-\e}$ to (\ref{0.2}) is actually equivalent to the initial problem.

The upper bound for $F(x)$ follows from Theorem 1 of \cite{Bar-Vekh}. We also note that the same upper bound can be derived from the inequality $\tau(n)\geq 2^{\omega(n)}$ (here $\omega(n)$ denotes the number of distinct prime divisors of $n$) and a bound for the number of $n\leq x$ with a given value of $\omega(n^2+1)$. For completeness, we provide this argument. 
 
We observe that the methods of this work can also be applied for other functions similar to $\tau(n)$. For instance, let $\tau_k(n)$ be the generalized divisor function, $\tau_k(n) = \sum_{n=d_1 d_2\ldots d_k}1$; then one can show that
$$ 
\sum_{p\leq x}\frac{1}{\tau_k(p-1)}\asymp_k x (\log x)^{\frac{1}{k}-2}.
$$
  
  
\textbf{Acknowledgements.} The work of Mikhail Gabdullin was supported by the Russian Science Foundation under grant no.19-11-00001, https://rscf.ru/en/project/19-11-00001/.
The work of Vitalii Iudelevich was supported by the
Theoretical Physics and Mathematics Advancement Foundation ``BASIS''. 
\section{Notation}

By $\varphi(n) = \#\{k \leq n: (k,n) = 1\}$ we denote the Euler totient function, and we use $P^{+}(n)$ and $P^{-}(n)$ for the least and the greatest prime divisors of a number $n>1$ respectively; by convention, $P^+(1)=0$ and $P^-(1)=\infty$. By $\pi(x)$ we denote the number of primes up to $x$, and  $\pi(x;q,a)$ stands for  the number of primes up to $x$ belonging to the arithmetic progression $a\pmod{q}$; we also set $R(x;q,a)=\pi(x;q,a)-\frac{\pi(x)}{\varphi(q)}$. The notation $f(x)\ll g(x)$, as well as $f(x) = O(g(x))$, means that $|f(x)|\leq C g(x)$ for some absolute constant $C>0$ and all possible values of $x$. We write $f(x)\asymp g(x)$, if $f(x)\ll g(x)\ll f(x)$, and write $f(x)\ll_k g(x)$ if we want to stress that the implied constant depends on $k$.

Now we recall some notation from the sieve methods. Let  $\mathcal{A}$ be a finite set of positive integers, and $\mathcal{P}$ be a finite subset of primes. Define $P = \prod\limits_{p \in \mathcal{P}} p$ and $S(\mathcal{A}, \mathcal{P}) = \#\{a \in \mathcal{A}: (a,P)=1 \}$. Let also $\mathcal{A}_d = \#\{a \in \mathcal{A}: a \equiv 0 \pmod d\}$. We assume that, for all $d|P$, 
\begin{equation}\label{cond_g}
\mathcal{A}_d = Xg(d)+r_d,
\end{equation}
where $g(d)$ is a multiplicative function such that $0<g(p)<1$ for $p \in \mathcal{P}$ and $g(p)=0$ for $p \not \in \mathcal{P}$.
 
Further, let the set $\mathcal{P}$ be divided into disjoint subsets $\mathcal{P}_1, \mathcal{P}_2, \ldots, \mathcal{P}_t$, and let $P_j= \prod\limits_{p \in \mathcal{P}_j} p$. Finally, let
$\{k_j\}_{r=1}^t$  be a sequence of even numbers. We set $V_j= \prod\limits_{p \in \mathcal{P}_j} (1-g(p))$,
 $L_j= \log V_j^{-1}$,
 $E = \sum\limits_{j=1}^{t} \frac{e^{L_j} (L_j)^{k_j+1}}{(k_j+1)!}$,
 $R = \sum\limits_{\substack{d_j|P_j\\ \omega(d_j) \leq k_j}} |r_{d_1 \ldots d_t}|$, and $R' = \sum\limits_{l=1}^t  \sum\limits_{\substack{d_j|P_j \\ \omega(d_j) \leq k_j, j \ne l \\ \omega(d_l) = k_l+1}} |r_{d_1 \ldots d_t}|$.
 
\section{Auxiliary results}

We need the following version of the Brun-Hooley sieve, which is due to Ford and Halberstam \cite{Ford-Halb2000}.

\begin{theorem}\label{th3.1}
Let \eqref{cond_g} be met. Then 
$$
S(\mathcal{A}, \mathcal{P}) \leq X \prod_{p \in \mathcal{P}} (1-g(p))e^{E} + R
$$
and
$$
S(\mathcal{A}, \mathcal{P}) \geq X \prod_{p \in \mathcal{P}} (1-g(p))(1-E) - R - R'.
$$
\end{theorem}
 
\begin{proof}
See \cite{Ford-Halb2000}.
\end{proof}
 
\begin{lem}\label{lem3.2}
Let $a \leq x^{1/40}$ be an even positive integer and 
$$
F_a(x) = \#\left\{ n \leq (x-1)/a: an+1\text{ prime}\textit{ and } P^{-}(n)>x^{{1}/{40}} \right\}.
$$
Then, for $x\geq x_0$, 
$$
F_a(x) \geq \frac{c_1\pi(x)}{\varphi(a)\log x}-R_1,
$$
where $c_1>0$ is an absolute constant and
$$
0\leq R_1\leq \sum_{d\leq x^{13/40}}|R(x;ad,1)| 
$$
\end{lem}

\begin{rem}
Analogous upper bounds for $F_a(x)$ without the error term $R_1$ can be obtained by sifting the set $\{(an+1)(n+P): n\leq (x-1)/a\}$, where $P=\prod_{p\leq z}p$. 
\end{rem}

\begin{proof}
We apply Theorem \ref{th3.1} to the sets $\mathcal{A} = \{n \leq (x-1)/a: an + 1\, \text{ prime}\}$ and $\mathcal{P} = \{p \leq z\}$, where $z=x^{1/40}$ (the other parameters will be chosen later). Then $F_a(x)=S(\mathcal{A}, \mathcal{P})$, and, for any $d|P=\prod_{p\leq z}p$, we have
$$
\mathcal{A}_d =\#\{k\leq (x-1)/(ad): adk+1\, \text{ prime}\} 
= \pi(x; ad,1)=\frac{\pi(x)}{\varphi(ad)}+R(x; ad,1).
$$
Now we write a number $d|P$ in the form $d=d_1d_2$, where $(d_1,a)=1$ and $d_2$ consists only of primes which divide $a$. Then, since $\varphi(ad_2)=ad_2\prod_{p|a}(1-1/p)=d_2\varphi(a)$, we find that
$$
\frac{1}{\varphi(ad)}=\frac{1}{\varphi(d_1)\varphi(ad_2)}=\frac{1}{\varphi(d_1)d_2\varphi(a)},
$$
and, hence, the equality (\ref{cond_g}) holds with $X = \frac{\pi(x)}{\varphi(a)}$, the multiplicative function $g$ defined by 
$$
g(p) = \begin{cases}
\frac{1}{p-1},\ \quad \text{for}\ (p,a)=1, \\
\frac{1}{p},\,\quad \quad \text{otherwise},
\end{cases}
$$
and
$$
r_d = R(x; ad, 1).
$$
Besides, we have
$$ 
\prod_{p \in \mathcal{P}}(1-g(p))\asymp \frac{1}{\log x},
$$
with an absolute implied constant.
	
Now we choose the partition of $\mathcal{P}$ and define numbers $\{k_j\}_{j=1}^t$. Set $z_j= z^{2^{1-j }}$ and
$$
\mathcal{P}_j = \mathcal{P} \cap (z_{j+1}, z_j],
$$
where $t$ is the unique positive integer with $z_{t+1} < 2 \leq z_t$. Let $k_j = b + 2(j-1)$, where $b \geq 2$ is even and will be chosen later. We also set
$$
C=\prod_{p>2}\left(1+\frac{1}{p^2-2p}\right)\leq 1.52.
$$ 
By \cite{Ros}, Theorem 7 (see (3.26), (3.27)), for any $x>1$, 
\begin{equation}\label{3.1}
\frac{e^{-\gamma}}{\log x}\left(1-\frac{1}{\log^2x}\right)  < \prod_{p\leq x}\left(1-1/p\right) < \frac{e^{-\gamma}}{\log x}\left(1+\frac{1}{2\log^2x}\right),
\end{equation}
where $\gamma$ is Euler's constant. Hence, for any $z\geq \sqrt2$, 
\begin{equation}\label{3.2}
\prod_{z<p\leq z^2}(1-1/p)^{-1}\leq 3;
\end{equation} 
indeed, it follows from (\ref{3.1}) for $z\geq 4$, and can be checked manually for $\sqrt{2}\leq z<4$. Thus
$$
V_j^{-1} = \prod_{z_{j+1} < p \leq z_j} (1 - g(p))^{-1} \leq  C\prod_{z_{j+1}<p\leq z_j}(1-1/p)^{-1} \leq 3C\leq 5.
$$
Hence $L_j =  \log V_j^{-1}\leq L=\log 5$ and
$$
E = \sum_{j=1}^t \frac{e^{L_j} L_j^{k_j+1}}{(k_j+1)!} \leq e^L\sum_{j=1}^t\frac{L^{b+2j-1}}{(b+2j-1)!}. 
$$
	
Now we estimate $R+R'$. If a number $d$ corresponds to a summand from $R$, we have $d=d_1 \ldots d_t$ with $d_j | {P}_j$ and $\omega(d_j) \leq k_j=b+2(j-1)$. Therefore,
$$
d \leq z_1^{k_1} \ldots z_t^{k_t} \leq z^{b+(b+2)/2+(b+4)/4+...}=z^{2b+4}.
$$
If a number $d$ corresponds to a summand from $R'$, we similarly find that $d \leq z^{2b+4} z = z^{2b+5}$.
Since the numbers $d$ corresponding to the sums $R$ and $R'$ are distinct and do not exceed $z^{2b+5}$, we see that
$$
R+R' \leq \sum_{d \leq z^{2b+5}} |r_d|.
$$
Now we take $b=4$. Then $2b+5=13$, 
$$
E\leq 5\sum_{j=1}^{\infty}\frac{(\log 5)^{2j+3}}{(2j+3)!} \leq 0.48,
$$
and
$$
R+R' \leq \sum_{d \leq x^{13/40}} |R(x;ad,1)|.
$$
The claim follows.  
\end{proof}
 
Define
$$
\mathcal{M} = \left\{ a \geq 1: a \mid n^2 + 1 \text{ for some }  n \geq 1 \right\}.
$$ 
It is well-known that $a \in \mathcal{M}$ if and only if $4 \nmid a$ and $p \nmid a$ for any $p \equiv 3\, (\text{mod}\, 4)$.

\begin{lem}\label{lem3.3}
Let $2\leq a, z\leq x^{1/30}$, $a \in \mathcal{M}$ and  $P^{+}(a)\leq z$. Let 
$$ 
W_a(x,z) = \# \left\{ n \leq x: a|(n^2 + 1) \text{ and } P^-\left(\frac{n^2+1}{a}\right) > z  \right\}.
$$
Then
$$
W_a(x,z) \asymp \frac{2^{\omega(a)}}{\varphi(a)}\frac{x}{\log z} .
$$
\end{lem}
 
\begin{proof} Setting $\mathcal{A} = \{k: ak=n^2+1 \mbox{ for some } n\leq x\}$ and 
$$
\mathcal{P} = \begin{cases}
\{p \leq z: p\equiv 1\pmod{4}\}, &\text{ if $a$ is even};\\
\{2\}\cup \{p \leq z: p\equiv 1\pmod{4}\}, &\text{ if $a$ is odd,}
\end{cases}
$$ 
we have $W_a(x,z)  = S(\mathcal{A}, \mathcal{P})$. As before, we write each $d|P=\prod_{p\in\mathcal{P}}p$ as $d=d_1d_2$, where $(d_1,a)=1$ and $d_2$ consists only of primes dividing $a$.
Suppose that one of the numbers $d$ and $a$ is even. Then by the Chinese remainder theorem and the fact that the congruence $x^2+1\equiv0\pmod{p}$ has two solutions $p\equiv 1\pmod{4}$ and one solution for $p=2$, we get
\begin{multline*}
\mathcal{A}_d =\#\{n\leq x: n^2+1 \equiv 0\!\!\pmod{ad} \} = \frac{x2^{\omega(ad)-1}}{ad}+O(2^{\omega(ad)})\\
=\frac{x2^{\omega(a)+\omega(d_1)-1}}{ad_1d_2}+O(2^{\omega(ad)}).
\end{multline*}
If both $a$ and $d$ are odd, a similar argument shows that
$$ 
\mathcal{A}_d = \frac{x2^{\omega(a)+\omega(d_1)}}{ad_1d_2}+O(2^{\omega(ad)}).
$$ 
In both cases we have
\begin{equation*}
\mathcal{A}_d = \frac{x2^{\omega(a)+\omega(d_1)-\I(2|ad)}}{ad_1d_2}+O(2^{\omega(ad)})
 =\frac{x2^{\omega(a)-\I(2|a)}}{a}\dfrac{2^{\omega(d_1)-\I(2|d_1)}}{d_1d_2}+O\left(\tau(ad)\right), 
\end{equation*}
where $\I(2|l)$ is equal to one if $l$ is even and equals to zero otherwise. Thus we see that the condition \eqref{cond_g} holds with $X= \frac{x2^{\omega(a)-\I(2|a)}}{a}$, the multiplicative function $g$ defined on the primes from $\mathcal{P}$ by
$$
g(p) = \begin{cases}
\frac{2}{p}, \quad p\nmid a, \\
\frac{1}{p}, \quad p|a,
\end{cases}
$$
(and also $g(2)=1/2$ in the case of odd $a$), and
$$
r_d = O(\tau(ad)).
$$
It is well-known that
\begin{equation}\label{3.3}
\prod\limits_{\substack{p\leq x \\ p\equiv\, 1(\text{mod}\,4)}}\left(1-\dfrac{1}{p} \right)\asymp \dfrac{1}{\sqrt{\log x}},
\end{equation}
(see \cite{Uchi}). Thus, in both cases we have
$$ 
\prod_{p \in \mathcal{P}}(1-g(p))\asymp \prod_{p\in \mathcal{P}, \, p>2}(1-2/p)\prod_{p|a,\,p>2}\frac{1-1/p}{1-2/p} \asymp \frac{a}{\varphi(a)\log z}
 $$
with an absolute implied constant.
 	
Now we choose the partition of $\mathcal{P}$ and define the numbers $\{k_j\}_{j=1}^t$. We again set $z_j= z^{2^{1-j }}$ and
$$
\mathcal{P}_j = \mathcal{P} \cap (z_{j+1}, z_j],
$$
where $t$ is a unique positive integer such that $z_{t+1} < 2 \leq z_t$ and $k_j = b + 2(j-1)$ with some even $b \geq 2$. Now (\ref{3.2}) implies that
 \begin{multline*}
 V_j^{-1} = \prod_{z_{j+1} < p \leq z_j} (1 - g(p))^{-1} \leq 2\prod_{z_{j+1}<p\leq z_j, \, p\neq2}(1-2/p)^{-1}\\
 =2\prod_{z_{j+1}<p\leq z_j, \, p\neq2}(1-1/p)^{-2}\frac{(1-1/p)^2}{(1-2/p)} \leq  18C\leq 28.
\end{multline*}
Therefore, $L_j =  \log V_j^{-1}\leq L=\log 28$ and
$$
E = \sum_{j=1}^t \frac{e^{L_j} L_j^{k_j+1}}{(k_j+1)!} \leq e^L\sum_{j=1}^t\frac{L^{b+2j-1}}{(b+2j-1)!}.
$$
Now we take $b=10$; then $E\leq 0.43$ and $2b+5=25$. As in the proof of Lemma \ref{lem3.2}, we get
$$
R+R' \leq \sum_{d \leq z^{2b+5}} |r_d|\ll \tau(a)\sum_{d\leq z^{25}}\tau(d)\ll x^{1/30}z^{25} \ll x^{26/30},
$$
and the main term is at least
$$
X\prod_{p\in \mathcal{P}}(1-g(p)) \asymp\frac{x2^{\omega(a)}}{\varphi(a)\log z} \gg\frac{x^{1-1/30}}{\log x}\,.
$$
Applying Theorem \ref{th3.1} concludes the proof. 
\end{proof}
  
The next lemma is a special case of the result of J.~Tenenbaum \cite{Tenen}. For completeness, we provide the proof.
 
\begin{lem}\label{lem3.4} 
Let $R>0$ be fixed. There exist positive constants $A=A(R)$ and $B=B(R)$ such that, for all $1\leq k \leq R\log\log x$,
$$
\#\{n\leq x: \o(n^2+1)=k\} 	\leq \frac{Ax(\log\log x+B)^{k-1}}{(k-1)! \log x}.
$$
\end{lem}	
 
\begin{proof}
For $m\geq 2$ and $y\geq2$, we define ``the $y$-smooth part'' of $m$: 
$$
d(m,y)=\max\{d|m: P^+(d)\leq y \}.
$$
For each $n\leq x$, the number $n^2+1$ can be written as the product $ab$, where
$$	
a=a(n)=\max\{d(n^2+1,y): d(n^2+1,y)\leq x^{1/30} \} 	
$$	
and $b=b(n)=(n^2+1)/a$; note that $(a,b)=1$. Now we set
$$
A_1=\{ n\leq x: a\leq x^{1/60}, P^-(b)>x^{1/60} \mbox{ and } b>1\},
$$
$$
A_2=\{ n\leq x: a\leq x^{1/60}, P^-(b)\leq x^{1/60}  \mbox{ and } b>1\},
$$
and
$$
A_3=\{ n\leq x: x^{1/60}< a \leq x^{1/30} \mbox{ and } b>1 \}.
 $$
Since the number of $n\leq x$ for which $b(n)=1$ does not exceed $x^{1/60}$, we conclude that
\begin{equation}\label{3.4}
\#\{n\leq x: \o(n^2+1)=k\} = N_1+N_2+N_3+O(x^{1/60}),	
\end{equation}
where
$$
N_i = \#\{n\in A_i: \o(n^2+1)=k \}, \quad i=1, 2, 3. 
$$
 	
First, we estimate $N_1$. In this case $1\leq \o(b)\leq 120$; by Lemma {\ref{lem3.3}}, we have 

\begin{multline}\label{3.5}
N_1 	\leq \sum_{l=k-120}^{k-1}\sum_{\substack{a\leq x^{1/60}, \, a\in \M \\ \o(a)=l }} W_a(x, x^{1/60}) \ll \sum_{l=k-120}^{k-1}\sum_{\substack{a\leq x^{1/60}, \, a\in \M \\ \o(a)=l }}\frac{x2^{\o(a)}}{\varphi(a)\log x} \\
\leq \frac{x}{\log x}\sum_{l=k-120}^{k-1}2^l\sum_{\substack{a\leq x^{1/60}, \, a\in \M \\ \o(a)=l }}\frac{1}{\varphi(a)}.
\end{multline}
Taking the logarithms of both sides of (\ref{3.3}), we get
$$
\sum_{\substack{p\leq x \\ p\equiv 1(\text{mod}\,4)}}\dfrac{1}{p} = \frac12\log\log x+O(1).
$$
Hence, for each $l\geq 0$, the inner sum in (\ref{3.5}) is at most
$$
\frac{1}{l!}\left(\sum_{p\leq x^{1/60},\, p\notequiv 3\!\!\pmod{4}}\frac{1}{\varphi(p)}+\frac{1}{\varphi(p^2)}+...  \right)^l \leq \frac{1}{l!}\left(0.5\log\log x+B_1\right)^l
$$
for some absolute constant $B_1>0$. Thus
\begin{equation}\label{3.6}
 N_1\ll \frac{x(\log\log x+B_1)^{k-1}}{(k-1)!\log x}\left(1+R+...+R^{120} \right).  
\end{equation}
 	
Now we estimate $N_2$. Let $p=P^-(b)$, and let $r$ be the largest integer such that $p^r|b$. Then, by definition of $a$, we have $ap^r>x^{1/30}$. Thus, for any $n\in A_2$ we have $p^r> x^{1/60}$, and since $p\leq x^{1/60}$, we see that $r\geq 2$. Let $\nu = \min\{u\geq 1: p^u>x^{1/60}\}$. Then $2\leq \nu\leq r$ and $p^{\nu-1}\leq x^{1/60}$. Hence, $p^\nu\leq x^{1/60}p\leq x^{1/30}$. We conclude that, for each $n\in A_2$, the number $n^2+1$ is divided by $p^\nu$, where $p$ is a prime such that $x^{1/60}<p^\nu\leq x^{1/30}$ and $\nu\geq 2$. Therefore, 
$$
N_2\leq\sum_{x^{1/60}<p^\nu\leq x^{1/30},\,\nu\geq 2}\left( \frac{2x}{p^\nu}+O(1)\right)\ll \sum_{p^\nu>x^{1/60},\,\nu\geq 2}\frac{x}{p^\nu}.
$$
Note that $p^{\nu} \geq \max\{p^2, x^{1/60}\}$, and therefore, $p^{\nu}\geq px^{1/120}$. It follows that
\begin{equation}\label{3.7}
	N_2\ll \sum_{p\leq x^{1/30}}\frac{x}{px^{1/120}} \ll x^{1-1/120}\log\log x.
\end{equation}
 	
Finally, let us consider $N_3$. Setting $q=P^+(a)$, we have $P^-(b)\geq q+1$ and 
\begin{equation}\label{3.8}
\o(b)\leq \frac{\log (x^2+1)}{\log (q+1)}\leq \frac{2\log x}{\log q}=:\eta.
\end{equation}
Lemma \ref{lem3.3} implies that
\begin{multline}\label{3.9} 
N_3 \leq \sum_{\substack{q\leq x^{1/30} \\ q\not{\equiv} 3\!\!\pmod{4}}}\sum_{k-\eta\leq l\leq k-1}	\sum_{\substack{x^{1/60}<a\leq x^{1/30} \\ P^+(a)=q, \, \o(a)=l}}W_a(x,q) \\
\ll x\sum_{\substack{q\leq x^{1/30} \\ q\not{\equiv} 3\!\!\pmod{4}}}\frac{1}{\log q}\sum_{k-\eta\leq l\leq k-1}2^l	\sum_{\substack{x^{1/60}<a\leq x^{1/30} \\ P^+(a)=q, \, \o(a)=l}}\frac{1}{\varphi(a)}.
\end{multline}
Let
\begin{equation}\label{3.10}
\d=\frac{C}{\log q},  
\end{equation}
where $C=120\log(R+2)+60$. Let us denote by $N_3^{(1)}$ and $N_3^{(2)}$ the contribution of $q\leq e^{2C}$ and of $q>e^{2C}$, respectively, to the sum in (\ref{3.9}). Since $\varphi(a)\gg\frac{a}{\log\log a}\gg\frac{x^{1/60}}{\log\log x}$, we have
\begin{equation}\label{3.11} 
N_3^{(1)}\ll x^{59/60}\log\log x\sum_{q\leq e^{2C}}\sum_{1\leq l\leq \pi(q)}2^l\sum_{\substack{a\leq x^{1/30} \\ P^+(a)=q}}1 \ll x^{59/60}(\log x)^{c_R},
\end{equation}
where $c_R>0$ depends only on $R$. Now we turn to $N_3^{(2)}$. For fixed $q$ and $l$, we see that
\begin{multline*}
S_{q,l} := \sum_{\substack{x^{1/60}<a\leq x^{1/30} \\ P^+(a)=q, \, \o(a)=l}}\frac{1}{\varphi(a)} \leq \sum_{\substack{a\leq x^{1/30} \\ P^+(a)=q, \, \o(a)=l}}\left(\frac{a}{x^{1/60}}\right)^{\delta}\frac{1}{\varphi(a)} \\
\leq x^{-\d/60}\frac{1}{(l-1)!}\left(\sum_{\substack{p<q \\ p\not{\equiv} 3\!\!\!\!\pmod{4}}}\frac{p^{\d}}{\varphi(p)}+\frac{p^{2\d}}{\varphi(p^2)}+... \right)^{l-1}\left(\frac{q^{\d}}{\varphi(q)}+\frac{q^{2\d}}{\varphi(q^2)}+... \right).
\end{multline*}
Note that $\frac{q^{\d}}{\varphi(q)}+\frac{q^{2\d}}{\varphi(q^2)}+...\ll_R 1/q$ for $q>e^{2C}$, and the sum over primes $p$ is equal to (since $e^u\leq 1+O_K(u)$ for $0\leq u\leq K$)
$$
\sum_{\substack{p<q \\ p\not{\equiv} 3\!\!\pmod{4}}}\frac{p^\d}{p}+O_R(1) 
 \leq \sum_{\substack{p<q \\ p\not{\equiv} 3\!\!\pmod{4}}}\frac{1}{p}+
 O_R\left(\d\sum_{\substack{p<q \\ p\not{\equiv} 3\!\!\pmod{4}}}\frac{\log p}{p}+1\right)\leq 0.5\log\log x+B,
$$
where $B=B(R)>0$. Thus,
$$
S_{q,l}\ll_R \frac{x^{-\d/60}}{q(l-1)!}\left(0.5\log\log x+B\right)^{l-1}
$$ 
and
\begin{multline}\label{3.12}
N_3^{(2)}\ll_R x\sum_{\substack{e^{2C}<q\leq x^{1/30} \\ q\not{\equiv} 3(\text{mod}\,4)}}\dfrac{x^{-\delta/60}}{q\log q}\sum_{k-\eta\leq l\leq k-1}\dfrac{2^l(0.5\log\log x+B)^{l-1}}{(l-1)!}\\
\ll \frac{x(\log\log x+2B)^{k-1}}{(k-1)!}\sum_{q\leq x^{1/30}}\frac{x^{-\d/60}}{q\log q}\left(1+R+...+R^{[\eta]}\right).
\end{multline}
Using the inequality
$$
1+R+\ldots+R^{[\eta]}\leq (R+2)^{\eta+1}
$$
and the fact that
$$
(R+2)^\eta x^{-\delta/60} = x^{-1/\log q},
$$
(which follows from the definitions (\ref{3.8}) and (\ref{3.10}) of $\eta$ and $\delta$ respectively), we see that the sum over primes $q$ in (\ref{3.12}) is at most
\begin{multline*} 
(R+2) \sum_{q\leq x^{1/30}}\frac{(R+2)^{\eta}x^{-\d/60}}{q\log q} \leqslant (R+2)\sum_{q\leq x^{1/2}}\frac{x^{-1/\log q}}{q\log q} \\
\leq \frac{R+2}{\log x}\sum_{j\geq 1}\sum_{q\in \left(x^{2^{-(j+1)}}, \,x^{-2^j}\right]}\frac1q \,2^{j+1}\exp(-2^j) \ll_R \frac{1}{\log x}\,.
\end{multline*} 
Substituting this into (\ref{3.12}) and taking into account (\ref{3.11}), we find that
\begin{equation}\label{3.13}
N_3 \ll_R \frac{x(\log\log x+2B)^{k-1}}{(k-1)!\log x}.
\end{equation}

Combining (\ref{3.4}), (\ref{3.6}), (\ref{3.7}), and (\ref{3.13}) we complete the proof. 
\end{proof}

\section{Proof of Theorem \ref{th1.1}}
 
Writing each $p\leq x$ in the form $p-1=ab$, where $P^+(a)\leq x^{1/40}$ and $P^-(b)>x^{1/40}$, we can rewrite the sum $\Phi(x)$ as  
$$
\Phi(x) = \sum_{\substack{a \leq x \\ P^+(a) \leq x^{1/40}}} \frac{1}{\tau(a)} \sum_{\substack{b \leq \frac{x-1}{a} \\ P^-(b) > x^{{1}/{40}} \\ ab+1 \text{ prime}}}  \frac{1}{\tau(b)}.
$$
Let $b = p_1^{\alpha_1} p_2^{\alpha_2} \ldots p_s^{\alpha_s}$ be the prime factorization of $b$. Then
$$
x^{ (\alpha_1 + \alpha_2 + \ldots + \alpha_s)/40} < b \leq x,
$$
and, hence, $\alpha_1+\ldots+\alpha_s < 40$ and
$$
\tau(b)=(\alpha_1+1) \ldots (\alpha_s+1) \leq 2^{\alpha_1+\ldots+\alpha_s} < 2^{40}.
$$
Therefore,
$$
\Phi(x) \geq 2^{-40} \sum_{\substack{a \leq x \\ P^+(a) \leq x^{1/40}}} \frac{1}{\tau(a)} \sum_{\substack{b \leq \frac{x-1}{a} \\ P^-(b) > x^{1/40} \\ ab+1  \text{ prime}}} 1 \geq 
 2^{-40} \sum_{\substack{a \leq x^{1/40} \\ a \text{ even}}} \frac{1}{\tau(a)}F_a(x).
$$
Lemma \ref{lem3.2} implies that
\begin{equation}\label{4.1}
\Phi(x) \geq \frac{c_2 \pi(x)}{\log x}\sum_{\substack{a \leq x^{1/40} \\ a \text{ even}}}  \frac{1}{\tau(a)\varphi(a)}-R_2,
\end{equation}
where $c_2>0$ and
$$
0\leq R_2 \leq \sum_{a\leq x^{1/40}}\sum_{d\leq x^{13/40}}|R(x;ad,1)| \leq 
\sum_{q\leq x^{7/20}}\tau(q)|R(x;q,1)|. 
 $$
Using the trivial bound $|R(x;q,1)|\ll x/q$ and the Cauchy-Schwarz inequality, we get
$$
R_2 \ll x^{1/2} \sum_{q \leq x^{0.35}} \frac{\tau(q)}{q^{1/2}} (|R(x; q, 1)|)^{1/2}
\leq x^{1/2} \left(\sum_{q \leq x^{0.35}} \frac{\tau^2(q)}{q}\right)^{1/2} 
\left(\sum_{q \leq x^{0.35}}  |R(x; q, 1)|\right)^{1/2}.
$$
Applying the bound
$$
\sum_{n\leq x}\frac{\tau^2(n)}{n} \asymp (\log x)^4,
$$ 
which follows from $\sum_{n\leq x}\tau^2(n)\asymp x(\log x)^3$ (see \cite{Prah}, Chapter III, Exercise 7) by partial summation and the Bombieri–Vinogradov theorem, we find that, for any fixed $A>0$,
$$
R_2 \ll_A x^{1/2}(\log x)^2 \frac{x^{1/2}}{(\log x)^{A/2}} = \frac{x}{(\log x)^{A/2-2}}.
$$
By taking $A = 24$ (say), we obtain
\begin{equation}\label{4.2}
R_2 \ll \frac{x}{(\log x)^{10}}.
\end{equation}
Now we turn to the sum
$$
T_1=\sum_{\substack{a \leq x^{1/40} \\ a \text{ even}}} \frac{1}{\tau(a) \varphi(a)} = \sum_{l \leq 0.5x^{1/40}} \frac{1}{\tau(2l)\varphi(2l)}.
$$
Since $\tau(mn) \leq \tau(m)\tau(n)$ and $\varphi(n) \leq n$, we see that
$$
T_1\gg \sum_{l \leq 0.5x^{1/40}} \frac{1}{\tau(l)l}.
$$
From (\ref{1.1}) and partial summation we have
$$
T_1\gg (\log x)^{1/2}.
$$
Combining this with (\ref{4.1}) and (\ref{4.2}), we obtain
$$
\Phi(x)\gg \frac{\pi(x)}{(\log x)^{1/2}}-O\left( \frac{x}{(\log x)^{10}}\right)\gg \frac{x}{(\log x)^{3/2}}.
$$
Now the claim follows.
 
\section{Proof of Theorem \ref{th1.2}}

We first prove the lower bound. For any $z\geq2$,
$$
F(x) = \sum_{n \leq x} \dfrac{1}{\tau(n^2+1)} = \underset{P^+(a) \leq z,\, P^-(b) > z} {\underset{ab = n^2 + 1,\,n \leq x}{\sum\sum}} \frac{1}{\tau(ab)}.
$$
Setting $z = x^{1/30}$, we have $\tau(b) \leq 2^{60}$, and now by Lemma \ref{lem3.3},
\begin{multline*}
F(x) \geq 2^{-60} \sum_{\substack{a \leq x^{1/30} \\ a \in \mathcal{M}}}  \frac{1}{\tau(a)} W_a\left(x, x^{1/30}\right)\gg 
\frac{x}{\log x} \sum_{\substack{a \leq x^{1/30} \\ a \in \mathcal{M}}}  \frac{2^{\o(a)}}{\tau(a)\varphi(a)} \\ 
 \geq  \frac{x}{\log x}\!\!\!\!\!\sum_{\substack{a \leq x^{1/30} \\ a \in \mathcal{M} \\ a  \text{ square-free}}} \!\!\!\!\!\frac{1}{a}  =  \frac{x}{\log x}\sum_{\substack{a \leq x^{1/30} \\ a \in \mathcal{M}}} \frac{1}{a}\,\sum_{\delta^2|a}\mu(\d).
 \end{multline*}
Changing the order of summation and using the fact that $\sum_{n=1}^{\infty}\frac{1}{n^2} = \frac{\pi^2}{6}$, we have
\begin{equation}\label{5.1}
 F(x)\gg \frac{x}{\log x}\sum_{\substack{\d \leq x^{1/60} \\ \d \in \mathcal{M}}}  \frac{\mu(\d)}{\d^2} \sum_{\substack{a \leq x^{1/30}\d^{-2} \\ a \in \mathcal{M} }}  \frac{1}{a} \geq \left( 2-\frac{\pi^2}{6}\right)\frac{x}{\log x}\sum_{\substack{a \leq x^{1/30} \\ a \in \mathcal{M}}} \frac{1}{a}. 
 \end{equation}
It is well-known (see \cite{Land1909}, \S 183) that for $y\geq 2$ one has
$$
\# \{a\leq y: p|a \Rightarrow p\equiv 1\,(\text{mod\,} 4)\}  \asymp \frac{y}{(\log y)^{1/2}}.
$$
By partial summation,
$$
\sum_{\substack{a\leq y \\ p|a \Rightarrow p\equiv 1\,(\text{mod\,} 4)
}}\frac{1}{a}\asymp (\log y)^{1/2}.
$$
Now the claim follows from (\ref{5.1}). 

Now we prove the upper bound. By Lemma \ref{lem3.4},
\begin{multline*}
 F(x)\leq \sum_{n\leq x}\frac{1}{2^{\o(n^2+1)}}
 =\sum_{k\leq 2\log\log x}\frac{\#\{n\leq x: \o(n^2+1) =k\}}{2^k} + O\left(\sum_{k>2\log\log x}\frac{x}{2^k}\right)\\
 \ll \frac{x}{\log x}\sum_{k=1}^{+\infty}\frac{(0.5\log\log x+0.5B)^{k-1}}{(k-1)!} + O\left(\frac{x}{\log x}\right)\ll \frac{x}{(\log x)^{1/2}}.
\end{multline*}
This concludes the proof.

\end{document}